\documentclass[11pt]{article}
\usepackage{amsfonts,amsmath,amssymb,amsthm}
\usepackage[dvipsnames,svgnames]{xcolor}
\usepackage[colorlinks=true,linkcolor=RoyalBlue,urlcolor=RoyalBlue,citecolor=PineGreen]{hyperref}
\usepackage[nameinlink]{cleveref}
\usepackage{tikz}
\usepackage{ifthen}
\newtheorem{theorem}{Theorem}

\newtheorem{lemma}[theorem]{Lemma}
\newtheorem{claim}[theorem]{Claim}
\newtheorem{cor}[theorem]{Corollary}

\newtheorem{conj}[theorem]{Conjecture}

\newcommand{\rank}{{\rm  rank}}
\newcommand{\dist}{{\rm dist}}
\newcommand{\R}{{\mathbb R}}
\newcommand{\real}{{\mathbb R}}

\newcommand{\sm}{{\setminus}}

\newcommand{\scrr}{{\mathcal R}}
\newcommand{\scrb}{{\mathcal B}}
\newcommand{\scrm}{{\mathcal M}}

\colorlet{colG}{DarkSeaGreen}
\definecolor{colR}{HTML}{CC6677}
\definecolor{colO}{HTML}{DDCC77}
\definecolor{colB}{HTML}{6699CC}

\usetikzlibrary{patterns}
\usetikzlibrary{decorations.markings}
\usetikzlibrary{calc}
\tikzset{roundnode/.style={circle,draw=black!50,fill=black!20,inner sep=1.2pt}}
\tikzset{every loop/.style={}}

\tikzstyle{vertex}=[circle, draw, fill=black, inner sep=0pt, minimum size=4pt]
\tikzstyle{smallvertex}=[circle, draw, fill=black, inner sep=0pt, minimum size=2pt]
\tikzstyle{edge}=[line width=1.5pt,black!50!white]

\tikzstyle{genericgraph}=[dashed]
\tikzstyle{labelsty}=[font=\scriptsize]

\Crefname{enumi}{}{}

\date{}
\title{Flexible circuits in the $d$-dimensional rigidity matroid\footnote{MSC: 52C25 (primary) and 05C10 (secondary). Key-words and phrases: bar-joint framework, rigid graph, rigidity matroid, flexible circuit.}}
\author{Georg Grasegger\thanks{Johann Radon Institute for Computational and Applied Mathematics (RICAM), Austrian Academy of Sciences. E-mail: georg.grasegger@ricam.oeaw.ac.at}, Hakan Guler\thanks{Department of Mathematics, Faculty of Arts \& Sciences, Kastamonu University, Kastamonu, Turkey. E-mail: hakanguler19@gmail.com}, Bill Jackson\thanks{School of Mathematical Sciences, Queen Mary
University of London, Mile End Road, London E1 4NS, United Kingdom.
E-mail: b.jackson@qmul.ac.uk} ~and Anthony Nixon\thanks{Department of Mathematics and Statistics, Lancaster University, Lancaster, LA1 4YF, United Kingdom. E-mail: a.nixon@lancaster.ac.uk}}

\begin{document} \maketitle

\begin{abstract}
A bar-joint framework $(G,p)$ in $\mathbb{R}^d$ is rigid if the only edge-length preserving continuous motions of the vertices arise from isometries of $\mathbb{R}^d$. It is known that, when $(G,p)$ is generic, its rigidity depends only on the underlying graph $G$, and 
is determined by the rank of the edge set of $G$ in the generic $d$-dimensional rigidity matroid $\scrr_d$. Complete combinatorial descriptions of the rank function of this matroid are known when $d=1,2$, and imply that all circuits in $\scrr_d$ are generically rigid in $\mathbb{R}^d$ when $d=1,2$. Determining the rank function of $\scrr_d$ is a long standing open problem when $d\geq 3$,  and the existence of non-rigid circuits in $\scrr_d$ for $d\geq 3$ is a major contributing factor to why this problem is so difficult.
We begin a study of non-rigid circuits  by characterising the non-rigid circuits in $\scrr_d$ which have at most $d+6$ vertices.
\end{abstract}

\section{Introduction}

A bar-joint \emph{framework} $(G,p)$ in $\mathbb{R}^d$ is the combination of a finite graph $G=(V,E)$ and a realisation $p:V\rightarrow \mathbb{R}^d$. The framework is said to be \emph{rigid} if the only edge-length preserving continuous motions of its vertices arise from isometries of $\mathbb{R}^d$, and otherwise it is said to be {\em flexible}. The study of the rigidity of frameworks has its origins in the work of Cauchy and Euler on Euclidean polyhedra \cite{Ca} and Maxwell \cite{Ma} on frames. 

Abbot \cite{Abb} showed that it is  NP-hard to determine whether a given $d$-dimensional framework is rigid whenever $d\geq 2$. The problem becomes more tractable for generic frameworks $(G,p)$ since we can linearise the problem and consider `infinitesimal rigidity' instead. We define the  \emph{rigidity matrix} $R(G,p)$ as the $|E|\times d|V|$ matrix  in which, for $e=v_iv_j\in E$, the submatrices in row $e$ and columns $v_i$ and $v_j$ are $p(v_i)-p(v_j)$ and $p(v_j)-p(v_i)$, respectively, and all other entries are zero.
We say that $(G,p)$ is \emph{infinitesimally rigid} if $|V|\leq d+1$ and $\rank\, R(G,p)=\binom{|V|}{2}$ or $|V|\geq d+2$ and $\rank\, R(G,p)=d|V|-\binom{d+1}{2}$.
Asimow and Roth \cite{AR} showed that infinitesimal rigidity is equivalent to rigidity for generic frameworks (and hence that generic rigidity depends only on the underlying graph of the framework).

The {\em $d$-dimensional rigidity matroid} of a graph $G=(V,E)$ is the matroid $\scrr_d(G)$ on $E$ in which a set of edges $F\subseteq E$ is independent whenever the corresponding rows of $R(G,p)$ are independent, for some (or equivalently every) generic $p$. 
We denote the rank function of $\scrr_d(G)$ by $r_d$ and put $r_d(G)=r_d(E)$. We say that  $G$ is: \emph{$\scrr_d$-independent} if $r_d(G)=|E|$; \emph{$\scrr_d$-rigid} if $G$ is a complete graph on at most $d+1$ vertices or $r_d(G)=d|V|-\binom{d+1}{2}$; \emph{minimally $\scrr_d$-rigid} if $G$ is $\scrr_d$-rigid and $\scrr_d$-independent; and an \emph{$\scrr_d$-circuit} if $G$ is not $\scrr_d$-independent but $G-e$ is $\scrr_d$-independent for all $e\in E$.

It is not difficult to see that the 1-dimensional rigidity matroid of a graph $G$ is  equal to its cycle matroid.
Landmark results of Pollaczek-Geiringer \cite{laman,pol}, and  Lov\'asz and Yemini \cite{LY} characterise independence  and the rank function in  $\scrr_2$. These results imply that  every $\scrr_d$-circuit is rigid when $d= 1,2$. This is no longer true  when $d\geq 3$ (see Figures~\ref{fig:BgraphsGeneral} and \ref{fig:BgraphsExample}  below),
and the existence of flexible $\scrr_d$-circuits is a fundamental obstuction to obtaining a combinatorial characterisation of independence in $\scrr_d$. 

Previous work on flexible $\scrr_d$-circuits has concentrated on constructions, see Tay \cite{Tay}, and Cheng, Sitharam and Streinu \cite{CSS}.
We will adopt a different approach: that of characterising the flexible $\scrr_d$-circuits in which the number of vertices is small compared to the dimension. To state our theorem we need to define the following two families of graphs.

For $d\geq 3$ and $2\leq t\leq d-1$, 
the graph $B_{d,t}$ is defined by putting $B_{d,t}=(G_1\cup G_2)-e$ where $G_i\cong K_{d+2}$, $G_1\cap G_2\cong K_{t}$ and $e\in E(G_1\cap G_2)$. 
Note that 
the graph $B_{3,2}$  is the well known flexible $\scrr_3$-circuit,  commonly referred to as the ``double banana''.
The family  $\scrb_{d,d-1}^+$ consists of all graphs of the form  $
(G_1\cup G_2)-\{e,f,g\}$ where: $G_1\cong K_{d+3}$ and $e,f,g\in E(G_1)$;  $G_2\cong K_{d+2}$ and $e\in E(G_2)$; $G_1\cap G_2\cong K_{d-1}$;  $e,f,g$ do not all have a common end-vertex; if $\{f,g\}\subset E(G_1)\sm E(G_2)$ then $f,g$ do not have a common end-vertex.
See Figure~\ref{fig:BgraphsGeneral} for an illustration of the general construction and Figure~\ref{fig:BgraphsExample} for specific examples.

\begin{theorem}\label{thm:vertex} 
Suppose $G$ is a flexible $\scrr_d$-circuit with at most $d+6$ vertices. Then either
\begin{enumerate}
  \item $d=3$ and $G\in \{B_{3,2}\}\cup \scrb_{3,2}^+$ or
  \item  $d\geq 4$ and $G\in \{B_{d,d-1}$, $B_{d,d-2}\}\cup \scrb_{d,d-1}^+$.
\end{enumerate}
\end{theorem}

Theorem \ref{thm:vertex}
gives the following lower bound on the number of edges in a flexible $\scrr_d$-circuit. This is used in \cite{GJ} to obtain an upper bound on $r_d(G)$ for all $1\leq d\leq 11$.

\begin{cor}\label{cor:edges} Suppose $G=(V,E)$ is a flexible $\scrr_d$-circuit. Then 
$|E|\geq d(d+9)/2$, with equality if and only if $G=B_{d,d-1}$.
\end{cor}

Jord\'an \cite{Jor} characterises $\scrr_d$-rigid graphs with at most $d+4$ vertices. He suggests in \cite[Remark 1]{Jor} that it may be possible to extend the characterisation to graphs on more than $d+4$ vertices, but notes that the simple degree condition  given in his characterisation may not be sufficient because of the existence of the double banana.
Theorem \ref{thm:vertex} 
implies the following characterisation of $\scrr_d$-rigid graphs with at most $d+6$ vertices. Our characterisation is in terms of   {\em $d$-tight} subgraphs (which are defined in the next section). 

\begin{cor}\label{cor:tight}
Let $G=(V,E)$ be a graph with $d+1\leq |V|\leq d+6$. Then
$G$ is $\scrr_d$-rigid if and only if $G$ has a $d$-tight, $d$-connected spanning subgraph $H$ such that $B_{d,d-1},B_{d,d-2}\not\subseteq H$.
\end{cor}

\begin{figure}[ht]
  \begin{center}
    \begin{tikzpicture}
      \draw[genericgraph,pattern=north east lines,pattern color=black!10!white] (-0.6,0) circle[x radius=1.2cm, y radius=0.8cm];
			\draw[genericgraph,pattern=north west lines,pattern color=black!10!white] (0.6,0) circle[x radius=1.2cm, y radius=0.8cm];
			\node[vertex] (u1) at (0,0.4) {};
			\node[vertex] (u2) at (0,-0.4) {};
			\draw[edge,densely dashed,colR] (u1)to node[label={[labelsty,label distance=-0.1cm]right:$e$}] {} (u2);
			\draw[dashed,colR] decorate [decoration={markings,mark=at position 0.3cm with {\draw[colR,thick] (-3pt,-3pt) -- (3pt,3pt);\draw[colR,thick] (3pt,-3pt) -- (-3pt,3pt);}}] {(u1)--(u2)};
			\node[pin={[labelsty,pin distance=0.5cm]below:$d-1$ vertices}] at (0,-0.5) {};
			\node[pin={[labelsty,pin distance=0.75cm,align=center]above:$K_{d+2}-e$}] at (-1,0) {};
			\node[pin={[labelsty,pin distance=0.75cm,align=center]above:$K_{d+2}-e$}] at (1,0) {};
    \end{tikzpicture}
    \quad
    \begin{tikzpicture}
      \draw[genericgraph,pattern=north east lines,pattern color=black!10!white] (-0.6,0) circle[x radius=1.2cm, y radius=0.8cm];
			\draw[genericgraph,pattern=north west lines,pattern color=black!10!white] (0.6,0) circle[x radius=1.2cm, y radius=0.8cm];
			\node[vertex] (u1) at (0,0.4) {};
			\node[vertex] (u2) at (0,-0.4) {};
			\draw[edge,densely dashed,colR] (u1)to node[label={[labelsty,label distance=-0.1cm]right:$e$}] {} (u2);
			\draw[dashed,colR] decorate [decoration={markings,mark=at position 0.3cm with {\draw[colR,thick] (-3pt,-3pt) -- (3pt,3pt);\draw[colR,thick] (3pt,-3pt) -- (-3pt,3pt);}}] {(u1)--(u2)};
			\node[pin={[labelsty,pin distance=0.5cm]below:$d-2$ vertices}] at (0,-0.5) {};
			\node[pin={[labelsty,pin distance=0.75cm,align=center]above:$K_{d+2}-e$}] at (-1,0) {};
			\node[pin={[labelsty,pin distance=0.75cm,align=center]above:$K_{d+2}-e$}] at (1,0) {};
    \end{tikzpicture}
    \quad
    \begin{tikzpicture}[xscale=-1]
      \draw[genericgraph,pattern=north east lines,pattern color=black!10!white] (-0.8,0) circle[x radius=1.2cm, y radius=0.8cm];
			\draw[genericgraph,pattern=north west lines,pattern color=black!10!white] (0.8,0) circle[x radius=1.2cm, y radius=0.8cm];
			\node[vertex] (u1) at (0,0.4) {};
			\node[vertex] (u2) at (0,-0.4) {};
			\node[vertex] (v1) at (1.4,0.4) {};
			\node[vertex] (v2) at (1.7,-0.2) {};
			\node[vertex] (w1) at (0.6,-0.55) {};
			\node[vertex] (w2) at (1.2,-0.45) {};
			\draw[edge,densely dashed,colR] (u1)to node[label={[labelsty,label distance=-0.1cm]right:$e$}] {} (u2);
			\draw[edge,densely dashed,colR] (v1)to node[label={[labelsty,label distance=-0.1cm]left:$f$}] {} (v2);
			\draw[edge,densely dashed,colR] (w1)to node[label={[labelsty,label distance=-0.1cm]above:$g$}] {} (w2);
			\draw[dashed,colR] decorate [decoration={markings,mark=at position 0.3cm with {\draw[colR,thick] (-3pt,-3pt) -- (3pt,3pt);\draw[colR,thick] (3pt,-3pt) -- (-3pt,3pt);}}] {(u1)--(u2)};
			\draw[dashed,colR] decorate [decoration={markings,mark=at position 0.3cm with {\draw[colR,thick] (-3pt,-3pt) -- (3pt,3pt);\draw[colR,thick] (3pt,-3pt) -- (-3pt,3pt);}}] {(v1)--(v2)};
			\draw[dashed,colR] decorate [decoration={markings,mark=at position 0.3cm with {\draw[colR,thick] (-3pt,-3pt) -- (3pt,3pt);\draw[colR,thick] (3pt,-3pt) -- (-3pt,3pt);}}] {(w1)--(w2)};
			\node[pin={[labelsty,pin distance=0.5cm]below:$d-1$ vertices}] at (0,-0.5) {};
			\node[pin={[labelsty,pin distance=0.75cm,align=center]above:$K_{d+2}-e$}] at (-1,0) {};
			\node[pin={[labelsty,pin distance=0.75cm,align=center]above:$K_{d+3}-\{e,f,g\}$}] at (1,0) {};
    \end{tikzpicture}
  \end{center}
  \caption{Graphs $B_{d,d-1}$ on the left, $B_{d,d-2}$ in the middle and
  $G\in \scrb_{d,d-1}^+$ on the right.}
	\label{fig:BgraphsGeneral}
\end{figure}

\begin{figure}[h]
\begin{center}
\begin{tikzpicture}[font=\small]
	\foreach \i in {0,...,4}{
		\node[roundnode] at (36+72*\i:1cm) (\i) []{};
		\foreach \j in {0,...,\i}{
			\ifthenelse{\i=\j \OR \i=4 \AND \j=0}
				{}
				{\draw[] (\i)--(\j)};
		}
	}
	\begin{scope}[shift={(36:1cm)}]
		\begin{scope}[shift={(-36:1cm)}]
			\foreach \i in {5,...,7}{
				\node[roundnode] at ({144-72*(\i-4)}:1cm) (\i) []{};
				\draw[] (\i)--(0);
				\draw[] (\i)--(4);
				\foreach \j in {5,...,\i}{
					\ifthenelse{\i=\j \OR \i=4 \AND \j=0}
						{}
						{\draw[] (\i)--(\j)};
				}
			}
		\end{scope}
	\end{scope}

	\begin{scope}[xshift=4cm]
		\foreach \i in {0,...,5}{
			\node[roundnode] at (30+60*\i:1cm) (\i) []{};
			\foreach \j in {0,...,\i}{
				\ifthenelse{\i=\j \OR \i=5 \AND \j=0}
					{}
					{\draw[] (\i)--(\j)};
			}
		}
		\begin{scope}[shift={(30:1cm)}]
			\begin{scope}[shift={(-30:1cm)}]
				\foreach \i in {6,...,9}{
					\node[roundnode] at ({150-60*(\i-5)}:1cm) (\i) []{};
					\draw[] (\i)--(0);
					\draw[] (\i)--(5);
					\foreach \j in {5,...,\i}{
						\ifthenelse{\i=\j \OR \i=5 \AND \j=0}
							{}
							{\draw[] (\i)--(\j)};
					}
				}
			\end{scope}
		\end{scope}
	\end{scope}

	\begin{scope}[xshift=10cm,xscale=-1]
		\foreach \i in {0,...,4}{
			\node[roundnode] at (36+72*\i:1cm) (\i) []{};
			\foreach \j in {0,...,\i}{
				\ifthenelse{\i=\j \OR \i=4 \AND \j=0}
					{}
					{\draw[] (\i)--(\j)};
			}
		}
		\begin{scope}[shift={(36:1cm)}]
			\begin{scope}[shift={(-36:1cm)}]
				\foreach \i in {5,...,7}{
					\node[roundnode] at ({144-72*(\i-4)}:1cm) (\i) []{};
					\draw[] (\i)--(0);
					\draw[] (\i)--(4);
					\foreach \j in {5,...,\i}{
						\ifthenelse{\i=\j \OR \i=4 \AND \j=0 \OR \i=7 \AND \j=5}
							{}
							{\draw[] (\i)--(\j)};
					}
				}
				\node[roundnode] at (36:1.5cm) (8) []{}
					edge[] (0)
					edge[] (5)
					edge[] (6)
					edge[] (7);
			\end{scope}
		\end{scope}
	\end{scope}
\end{tikzpicture}
\end{center}
\caption{Graphs $B_{3,2}$ on the left, $B_{4,2}$ in the middle and 
$G\in \scrb_{3,2}^+$ on the right.}
\label{fig:BgraphsExample}
\end{figure}

We will prove Theorem \ref{thm:vertex} and Corollaries \ref{cor:edges} and \ref{cor:tight} in Section \ref{sec:main}.

\section{Preliminary Lemmas}
We will introduce some standard terminology and results from rigidity theory. We assume throughout this section that $d\geq 1$ is a fixed integer.
 
Given a vertex $v$ in a graph $G=(V,E)$, we will use $d_G(v)$ and $N_G(v)$ to denote the degree and neighbour set respectively of $v$. 
For a set $V'\subseteq V$, we put $N_G(V')=\left(\bigcup_{v\in V'} N_G(v)\right)-V'$.
We will use $\delta(G)$ and $\Delta(G)$ to denote the minimum and maximum degree, respectively, in $G$, and  $\dist_G(x,y)$ to denote the length of a shortest path between two vertices $x,y\in V$. We will suppress the subscript in these notations whenever the graph is clear from the context.
The graph $G$ is \emph{$d$-sparse} if $|E'| \leq d|V'|-\binom{d+1}{2}$ for all subgraphs $G'=(V',E')$ of $G$ with $|V'|\geq d+2$. It is {\em $d$-tight} if it is $d$-sparse and has $d|V|-\binom{d+1}{2}$ edges.
Our first result \cite[Lemma 11.1.3]{Wlong} shows that every $\scrr_d$-independent graph is $d$-sparse.

\begin{lemma}\label{lem:MW}
Let $G=(V,E)$ be $\scrr_d$-independent with $|V|\geq d+2$. Then $|E|\leq d|V|-\binom{d+1}{2}$.
\end{lemma}

The characterisations of $\scrr_d$-independence when $d\leq 2$ show that the converse of Lemma \ref{lem:MW} holds for these values of $d$. The existence of flexible $\scrr_d$-circuits implies that the converse fails for all $d\geq 3$.

A graph $G'$ is said to be obtained from another graph $G$ by: a \emph{($d$-dimensional) 0-extension}
if $G=G'-v$ for a vertex  $v\in V(G')$ with $d_{G'}(v)=d$; or a \emph{($d$-dimensional) 1-extension} if $G=G'-v+xy$ for a vertex $v\in V(G')$ with $d_{G'}(v)=d+1$ and $x,y\in N_{G'}(v)$.
The inverse operations of 0-extension and 1-extension are called \emph{0-reduction} and \emph{1-reduction}, respectively.

\begin{lemma}\label{lem:01ext}\cite[Lemma 11.1.1, Theorem 11.1.7]{Wlong}
Let $G$ be $\scrr_d$-independent and let $G'$ be obtained from $G$ by a 0-extension or a 1-extension. Then $G'$ is $\scrr_d$-independent.
\end{lemma}

We can use Lemma \ref{lem:01ext}  to show that an extension operation which adds a copy of $K_3$ preserves minimal rigidity.

\begin{lemma}\label{lem:rigunion}
Let $G=(V,E)$ be a graph, $\{V_1,V_2\}$ be a partition of $V$ and put $G_i=G[V_i]$ for $i=1,2$. Suppose $G_1$ is minimally $\scrr_d$-rigid, $G_2\cong K_3$,
each vertex of $G_2$ has $d-1$ neighbours in $G_1$ and the set of all neighbours of the vertices of $G_2$  in $G_1$ has size at least $d$.
Then $G$ is minimally $\scrr_d$-rigid.
\end{lemma}

\begin{proof}
Let $V(G_2)=\{x,y,z\}$ and $N_x, N_y, N_z$ denote the set of neighbours of $x,y,z$ in $G_1$, respectively. 
Since $|N_x\cup N_y\cup N_z|\geq d$ and $|N_x|=|N_y|=|N_z|=d-1$, at most two of the sets $N_x,N_y,N_z$ can be the same.
Therefore, we may assume that either the sets $N_x,N_y,N_z$ are all pairwise distinct, or $N_x=N_y\neq N_z$ (by relabelling if necessary).
This implies that the sets $N_z\setminus N_x$ and $N_y\setminus N_z$ are non-empty.
Then $G$ can be obtained from $G_1$ as follows. We first perform  a 0-extension which adds $x$ and edges from $x$ to its $d-1$ neighbours in $N_x$ as well as $w$ for some $w\in N_z\setminus N_x$. We next perform
a 1-extension which deletes $xw$, and adds $z$  and the edges from $z$ to its $d-1$ neighbours in $N_z$ as well as to $x$ and $u$ for some $u\in N_y\setminus N_z$. Finally we perform one more
1-extension which deletes $zu$ and adds $y$  and the edges from $y$ to its $d-1$ neighbours in $N_y$ as well as $x$ and $z$. See Figure \ref{fig:l9}.
Hence, $G$ is $\scrr_d$-independent by Lemma \ref{lem:01ext}.  Minimal $\scrr_d$-rigidity now follows by a simple edge count.
\end{proof}

\begin{figure}
\begin{center}
  \begin{tikzpicture}
      \draw[genericgraph,pattern=north east lines,pattern color=black!10!white] (-0.6,0) circle[x radius=0.8cm, y radius=1.5cm];
			\node[vertex,label={[labelsty]below:$x$}] (u2) at (1.5,-0.4) {};
			\node[vertex,label={[labelsty]left:$w$}] (w) at (-0.5,-0.2) {};
			\node[] (c2) at (-0.55,-1.3) {};
			\draw[edge] (u2)to node[labelsty,below] {$d-1$} (c2);
			\draw[edge] (u2)edge($(c2)+(0,0.25)$) (u2)edge($(c2)+(0,0.5)$) (u2)edge(w);
    \end{tikzpicture}
    \qquad
    \begin{tikzpicture}
      \draw[genericgraph,pattern=north east lines,pattern color=black!10!white] (-0.6,0) circle[x radius=0.8cm, y radius=1.5cm];
			\node[vertex,label={[labelsty]above:$z$}] (u1) at (1.5,0.4) {};
			\node[vertex,label={[labelsty]below:$x$}] (u2) at (1.5,-0.4) {};
			\node[vertex,label={[labelsty]left:$u$}] (u) at (-0.5,0.2) {};
			\node[] (c1) at (-0.65,1.4) {};
			\node[] (c2) at (-0.55,-1.3) {};
			\draw[edge] (u1)edge(u2);
			\draw[edge] (u1)to node[labelsty,above] {$d-1$} (c1);
			\draw[edge] (u1)edge($(c1)-(0,0.25)$) (u1)edge($(c1)-(0,0.5)$) (u1)edge(u);
			\draw[edge] (u2)to node[labelsty,below] {$d-1$} (c2);
			\draw[edge] (u2)edge($(c2)+(0,0.25)$) (u2)edge($(c2)+(0,0.5)$);
    \end{tikzpicture}
    \qquad
    \begin{tikzpicture}
      \draw[genericgraph,pattern=north east lines,pattern color=black!10!white] (-0.6,0) circle[x radius=0.8cm, y radius=1.5cm];
			\node[vertex,label={[labelsty]above:$z$}] (u1) at (1.5,0.4) {};
			\node[vertex,label={[labelsty]below:$x$}] (u2) at (1.5,-0.4) {};
			\node[vertex,rotate around=-60:(u1),label={[labelsty]60:$y$}] (u3) at (u2) {};
			\coordinate (w) at (-0.5,-0.2) {};
			\node[] (c1) at (-0.65,1.4) {};
			\node[] (c2) at (-0.55,-1.3) {};
			\draw[edge] (u1)edge(u2) (u2)edge(u3) (u3)edge(u1);
			\draw[edge] (u1)to node[labelsty,above] {$d-1$} (c1);
			\draw[edge] (u1)edge($(c1)-(0,0.25)$) (u1)edge($(c1)-(0,0.5)$);
			\draw[edge] (u2)to node[labelsty,below] {$d-1$} (c2);
			\draw[edge] (u2)edge($(c2)+(0,0.25)$) (u2)edge($(c2)+(0,0.5)$);
			\draw[edge] (u3)to node[labelsty,above,pos=0.15] {$d-1$} ($(w)+(0,0.25)$);
			\draw[edge] (u3)edge(w) (u3)edge($(w)-(0,0.25)$);
    \end{tikzpicture}
\end{center}
\caption{Construction of $G$ in the proof of Lemma \ref{lem:rigunion}.}
\label{fig:l9}
\end{figure}

A {\em ($d$-dimensional) vertex split} of a graph $G=(V,E)$ is the operation defined as follows: choose $v\in V$, $x_1,x_2,\dots,x_{d-1}\in N_G(v)$ and 
a partition $N_1,N_2$ of 
pairwise disjoint sets $N_1,N_2$ with $N_1\cup N_2=N_G(v)\setminus \{x_1,x_2,\dots,x_{d-1}\}$; then delete $v$ from
$G$ and add two new vertices $v_1,v_2$ joined to $N_1,N_2$, 
respectively; finally add new edges $v_1v_2,v_1x_1, v_2x_1,\break v_1x_2,v_2x_2,\dots, v_1x_{d-1},v_2x_{d-1}$.

\begin{lemma}\label{lem:vsplit}\cite[Proposition 10]{Wsplit}
Let $G$ be $\scrr_d$-independent and let $G'$ be obtained from $G$ by a vertex split. Then $G'$ is $\scrr_d$-independent.
\end{lemma}

Given a graph $G$, the \emph{cone} $G'$ of $G$ is the graph obtained from $G$ by adding a new vertex adjacent to every vertex of $G$.

\begin{lemma}\label{lem:coning}
Let $G'$ be the cone of a graph $G$. Then:
\begin{enumerate}
 \item\label{it:coning:rig} $G$ is $\scrr_d$-rigid if and only if $G'$ is $\scrr_{d+1}$-rigid \cite{Wcone};
 \item\label{it:coning:circ} $G$ is an $\scrr_d$-circuit if and only if $G'$ is an $\scrr_{d+1}$-circuit \cite{GGJ}.
\end{enumerate}
\end{lemma}

Our next two lemmas concern the operation of `gluing' two graphs together.

\begin{lemma}\label{lem:intbridge}\cite[Lemma 11.1.9]{Wlong}
Let $G_1$, $G_2$ be subgraphs of a graph $G$ and suppose that $G=G_1\cup G_2$.
\begin{enumerate}
\item\label{it:intbridge:rig} 
	If $|V(G_1)\cap V(G_2)|\geq d$ and $G_1,G_2$ are $\scrr_d$-rigid then $G$ is $\scrr_d$-rigid.
\item\label{it:intbridge:indep} 
	If $G_1\cap  G_2$ is $\scrr_d$-rigid and $G_1,G_2$ are $\scrr_d$-independent then $G$ is $\scrr_d$-independent.
\item\label{it:intbridge:rank} 
	If $|V(G_1)\cap V(G_2)| \leq d-1$, $u\in V(G_1)-V(G_2)$ and $v\in V(G_2)-V(G_1)$ then
	$r_d(G+uv)=r_d(G)+1$.
\end{enumerate}
\end{lemma}

Lemma \ref{lem:intbridge}\ref{it:intbridge:indep} immediately implies that every $\scrr_d$-circuit $G=(V,E)$ is $2$-connected and that,
if $G-\{u,v\}$ is disconnected for some $u,v\in V$, then $uv\not\in E$. Our next lemma gives more structural information for the case when $G-\{u,v\}$ is disconnected. 

Given three graphs $G=(V,E)$, $G_1=(V_1,E_1)$, and $G_2=(V_2,E_2)$, 
we say that
$G$ is a {\em $2$-sum of $G_1,G_2$ along an edge $e$} if $G=(G_1\cup G_2)-e$, $G_1\cap G_2=K_2$ and $e\in E_1\cap E_2$.
Our next result shows that the 2-sum of  $G_1,G_2$ is an $\scrr_d$-circuit if and only if $G_1,G_2$ are both $\scrr_d$-circuits. Its proof relies on the matroid circuit elimination axiom (which states that if $C_1,C_2$ are distinct circuits in a matroid $\scrm$ and $e\in C_1\cap C_2$ then $(C_1\cup C_2)-e$ contains a circuit of $\scrm$). 

\begin{lemma}\label{lem:2-sum} 
Suppose that $G=(V,E)$ is the $2$-sum of $G_1=(V_1,E_1)$ and $G_2=(V_2,E_2)$. Then $G$ is an $\scrr_d$-circuit if and only if $G_1$ and $G_2$ are both $\scrr_{d}$-circuits. 
\end{lemma}

\begin{proof}
We first prove necessity. Suppose that $G$ is an $\scrr_d$-circuit.
If  $G_1$ and $G_2$ are both  $\scrr_d$-independent then $G+uv$ is $\scrr_d$-independent by Lemma \ref{lem:intbridge}\ref{it:intbridge:indep}, a contradiction since $G$ is an $\scrr_d$-circuit.
If exactly one of $G_1$ and $G_2$, say $G_1$, is  $\scrr_d$-independent then $uv$ belongs to the unique $\scrr_d$-circuit contained in $G_2$.
We may extend $uv$ to a base of $E_i$, for $i=1,2$, and then apply Lemma \ref{lem:intbridge}\ref{it:intbridge:indep} to obtain $r_d(G+uv)=r_d(G_1)+r_d(G_2)-1$.
Thus we have $r_d(G)=r_d(G+uv)=|E_1|+|E_2|-2=|E|$, a contradiction since $G$ is an $\scrr_d$-circuit.  Hence
$G_1$ and $G_2$ are both $\scrr_d$-dependent. Then the matroid circuit elimination axiom  combined with the fact that $G$ is an $\scrr_d$-circuit imply that $G_1$ and $G_2$ are both $\scrr_d$-circuits.  

We next prove sufficiency. Suppose that $G_1$ and $G_2$ are both $\scrr_d$-circuits. The circuit elimination axiom implies that $G$ is $\scrr_d$-dependent and hence that $G$ contains an $\scrr_d$-circuit $G'=(V',E')$. Since $G_i-uv$ is $\scrr_d$-independent for $i=1,2$, we have $E'\cap E_i\neq \emptyset$. This implies that $G'$ is a 2-sum of $G_1'=(G_1\cap G')+uv$ and $G_2'=(G_2\cap G')+uv$. The proof of necessity in the previous paragraph now tells us that $G_1'$ and $G_2'$ are both $\scrr_d$-circuits. Since $G_i$ is an $\scrr_d$-circuit and $G_i'\subseteq G_i$ we must have  $G_i'=G_i$ for $i=1,2$ and hence $G=G'$. 
\end{proof}

The special cases of Lemma \ref{lem:2-sum} when $d=2,3$ were proved by  Berg and Jord\'an \cite{BJ} and Tay \cite{Tay}, respectively.

\medskip

We next obtain some results on the graphs in $\{B_{d,d-1}\}\cup \{B_{d,d-1}\}\cup \mathcal B_{d,d-1}^+$.
The {\em ($d$-dimensional) degree of freedom} of a graph $G=(V,E)$ with $|V|\geq d+1$ is defined to be the number $d|V|-\binom{d+1}{2}-r_d(G)$, i.e.~the minimum number of edges we need to add to  $G$ to make it  $\scrr_d$-rigid.
 We may apply Lemma \ref{lem:2-sum} to the $\scrr_3$-circuit $K_5$  
to deduce that $B_{3,2}$ is an $\scrr_3$-circuit which has 18 edges and one degree of freedom.
The same argument applied to the $\scrr_4$-circuit $K_6$ implies that $B_{4,2}$ is an $\scrr_4$-circuit with 28 edges and three degrees of freedom.
We can now use Lemma~\ref{lem:coning}\ref{it:coning:circ} to deduce that $B_{d,d-1}$ is an $\scrr_d$-circuit with $\frac{d(d+9)}{2}$ edges and one degree of freedom
and that $B_{d,d-2}$ is an $\scrr_d$-circuit with $d(d+3)$ edges and three degrees of freedom, for all $d\geq 4$.
Similarly, using the fact that $K_{d+3}-\{f,g\}$ is a rigid $\scrr_d$-circuit when $f,g$ are non-adjacent,
we may apply Lemma \ref{lem:2-sum} to the $\scrr_3$-circuits $K_5$  and $K_6-\{f,g\}$, for two non-adjacent edges $f,g$,
to deduce that every graph in $\mathcal B_{3,2}^+$ is an  $\scrr_3$-circuit with 21 edges and one degree of freedom. We can then use Lemma \ref{lem:coning}\ref{it:coning:circ}
to deduce that  if a graph from $\mathcal B_{d,d-1}^+$ is obtained by coning, it is an $\scrr_d$-circuit unless $f$ or $g$ has an end-vertex in $V_1\cap V_2$. Our next result extends this to all graphs in $\scrb_{d,d-1}^+$.
Note that, since every graph in $\scrb_{d,d-1}^+$ has one more vertex and $d$ more edges than  $B_{d,d-1}$, the fact that the graphs  in $\scrb_{d,d-1}^+$ are
$\scrr_d$-circuits will imply that they each have $\frac{d(d+9)}{2}+d$ edges and one degree of freedom.

\begin{lemma}\label{lem:banana}
Every graph in $\{B_{d,d-1},B_{d,d-2}\}\cup \scrb_{d,d-1}^+$ is an $\scrr_d$-circuit.
\end{lemma}

\begin{proof}
We have already seen that $B_{d,d-1}$ and $B_{d,d-2}$ are $\scrr_d$-circuits.
Let 
$G\in \scrb_{d,d-1}^+$
and suppose that $G_1=(V_1,E_1)$, $G_2=(V_2,E_2)$, and $e,f,g$ are as in the definition of $\scrb_{d,d-1}^+$. 
Since 
$G$ has $d|V(G)|-\binom{d+1}{2}$ edges and is
not  $\scrr_d$-rigid (since it is not $d$-connected),
it is $\scrr_d$-dependent.

We will complete the proof by showing that 
$G-h$ is $\scrr_d$-independent for all edges $h$ of 
$G$. If $h$ is incident with a vertex $x\in V_2\setminus V_1$,
then we can reduce 
$G-h$ to $G_1-\{e,f,g\}$ by recursively deleting vertices of degree at most $d$ (starting from $x$).  Since $G_1-\{e,f,g\}$ is $\scrr_d$-independent, Lemma \ref{lem:01ext} and the fact that edge deletion preserves independence now imply that 
$G-h$ is $\scrr_d$-independent.
Thus we may assume that $h\in E_2$.

Suppose that $f,g,h$ do not have a common end-vertex.  Choose a vertex $x\in V_2\setminus V_1$ and let
$H=G-h-x+e$ be the graph  obtained by applying a 1-reduction at $x$.
We can reduce $H$ to $G_1-\{f,g,h\}$ by recursively deleting vertices of degree at most $d$. Since  
$f,g,h$ do not have a common end-vertex, $G_1-\{f,g,h\}$ is $\scrr_d$-independent. We can now use Lemma \ref{lem:01ext} to deduce that
$G-h$ is $\scrr_d$-independent.

Hence we may assume that $f,g,h$ have a common end-vertex $u$. 
The definition of $\scrb_{d,d-1}^+$ now implies that at least one of $f$ and $g$, say $f$, is an edge of $G_1\cap G_2$.
Since $e, f,g$ do not have a common end-vertex, $e$ is not incident with $u$ and hence $e,g,h$ do not have a common end-vertex. 
We can now apply the argument in the previous paragraph with  the roles of $e$ and $f$ reversed to deduce that
$G-h$ is $\scrr_d$-independent.
\end{proof}

\begin{lemma}\label{lem:banana_1ext}
Let $G$ be a graph obtained from $B_{d,d-1}$ by a 1-extension operation.
Then either $G$ is $\scrr_d$-rigid or $G\in \scrb_{d,d-1}^+$.
\end{lemma}
\begin{proof}
Let $v$ be the new vertex added by the 1-extension and consider $B_{d,d-1}=(G_1\cup G_2)-e$ where $G_i\cong K_{d+2}$, $G_1\cap G_2\cong K_{d-1}$ and $e\in E(G_1\cap G_2)$.
If $N_G(v)\subseteq V(G_i)$ for some $i=1,2$, then $G\in \scrb_{d,d-1}^+$.

Hence, we may assume that there exist vertices $v_1\in N_G(v)\cap (V(G_1)\sm V(G_2))$ and
$v_2\in N_G(v)\cap (V(G_2)\sm V(G_1))$.
Note that as $v_1$ and $v_2$ are on different sides of the cut
set $V(G_1)\cap V(G_2)$ of $B_{d,d-1}$, we have $v_1v_2\notin E(B_{d,d-1})$. Let $f$ be the edge of $B_{d,d-1}$ deleted by the 1-extension.
We may use Lemma \ref{lem:intbridge}\ref{it:intbridge:rank} to obtain
\begin{equation*}
r_d(B_{d,d-1}-f+v_1v_2)=r_d(B_{d,d-1}+v_1v_2)=r_d(B_{d,d-1})+1.
\end{equation*}
Since $B_{d,d-1}$ has one degree of freedom, this implies that $B_{d,d-1}-f+v_1v_2$ is $\scrr_d$-rigid. We may now use the fact that $G$ can be obtained from
$B_{d,d-1}-f+v_1v_2$ by a 1-extension operation on the edge $v_1v_2$ and Lemma \ref{lem:01ext} to conclude that
$G$ is $\scrr_d$-rigid.
\end{proof}

Our last two lemmas are rather technical results which we will need in our proof of Theorem \ref{thm:vertex}.

\begin{lemma}\label{lem:comp}
\begin{enumerate}
	\item\label{it:comp:6reg} 
	Every 6-regular graph on 10 vertices is $\scrr_4$-independent.
	\item\label{it:comp:12reg} 
	 Every 12-regular graph on 15 vertices is $\scrr_9$-independent.
\end{enumerate}
\end{lemma}

\begin{proof}
There are 21 6-regular graphs on 10 vertices (see OEIS \cite{OEIS} sequence A165627 for the count and references to lists for download).
The number of 12-regular graphs on 15 vertices is 17.
These can be obtained from the fact that the complement of a 12-regular graph on 15 vertices is a 2-regular graph on 15 vertices, i.\,e.\ a graph consisting of disjoint cycles.

Now we need to show that these graphs are indeed $\scrr_d$-independent in the stated dimensions.
We can do so with the help of any computer algebra system. 
For each graph, we choose a vector $p \in \mathbb{R}^{d|V|}$ and compute the rank of $R(G,p)$.
We know that as soon as we find a $p$ such that rank $R(G,p) =|E(G)|$,  we will have $\rank\, R(G,q) =|E(G)|$ for all generic $q$.
We did this by  taking a random choice for $p$ and checking that rank $R(G,p) =|E(G)|$. (Due to generic rigidity, almost every random choice will do.)
\end{proof}

\begin{lemma}\label{lem:deg23} 
Suppose that $G=(V,E)$ is a graph with $|V|\geq 11$, minimum degree two and maximum  degree three. Then there exist vertices $x,y\in V$ with $d(x)=2$, $d(y)=3$ and $\dist(x,y)\geq 3$.
\end{lemma}

\begin{proof} 
Assume $G=(V,E)$ is a counterexample to the lemma.
Choose a vertex  $v\in V$ of degree $2$. Then there are at most 6 vertices at distance 1 or 2 from $v$. Hence $G$ has at most 6 vertices of degree 3. Now choose a vertex $u\in V$ of degree 3. Each neighbour of $u$ is either a vertex of degree 2 which has at most one other neighbour of degree 2 or a vertex of degree 3 which has at most two other neighbours of degree 2. Therefore $G$ has at most 6 vertices of degree 2. 
If there does not exist 6 vertices of degree 3 in $G$ then the number of vertices of degree 3 in $G$ is at most 4 by parity, and we would have $|V|\leq 10$. Hence there are exactly 6 vertices of degree 3 and $v$ is adjacent to two vertices of degree 3. Since $v$ is an arbitrary vertex of degree two, every vertex of degree 2 is adjacent to two vertices of degree 3. Now choose $w$ to be a vertex of degree 3 at distance 2 from $v$ and a vertex $y\neq v$, of degree 2, not adjacent to $w$. Then $\dist(w,y)\geq 3$.
\end{proof}

\section{Main results}
\label{sec:main}

We will prove Theorem \ref{thm:vertex}, Corollary \ref{cor:edges} and Corollary  \ref{cor:tight}.

\subsection{Proof of Theorem \ref{thm:vertex}} 
We proceed by contradiction. Suppose the theorem is false and choose a counterexample $G=(V,E)$ such that $d$ is as small as possible and, subject to this condition, $|V|$ is as small as possible. Since all $\scrr_d$-circuits are $\scrr_d$-rigid when $d\leq 2$,
we have $d\geq 3$. Since $G$ is an $\scrr_d$-circuit, $G-v$ is $\scrr_d$-independent for all $v\in V$, and we can now use the fact that $0$-extension preserves $\scrr_d$-independence (by Lemma \ref{lem:01ext}) to deduce that 
$\delta(G)\geq d+1$. 
Since $G$ is a flexible $\scrr_d$-circuit, $G$ is $d$-sparse by Lemma \ref{lem:MW}.

\smallskip

\noindent \textbf{\boldmath
Case 1: $d(v)=d+1$ for some $v\in V$}.

Since  $G$ does not  contain the rigid $\scrr_d$-circuit $K_{d+2}$, $v$ has  two non-adjacent neighbours $v_1,v_2$.
If $H=G-v+v_1v_2$ was $\scrr_d$-independent then $G$ would be $\scrr_d$-independent by Lemma~\ref{lem:01ext}. Hence $H$ contains an $\scrr_d$-circuit $C$.
Since $C$ has at most $d+5$ vertices, the minimality of $G$ implies that either $C$ is $\scrr_d$-rigid, or $C=B_{d,d-1}$ and $C$ is a spanning subgraph of $H$. 
If the latter alternative occurs then Lemma \ref{lem:banana_1ext} would imply that $G$ contains a circuit $C'\in \mathcal{B}_{d,d-1}^+$ and we would contradict the choice of $G$.
Hence $C$ is $\scrr_d$-rigid and $G$ contains the minimally $\scrr_d$-rigid subgraph $C-v_1v_2$. Since $C$ is a rigid $\scrr_d$-circuit, we have $|V(C)|\geq d+2$.
Let $G'$ be a 
minimally $\scrr_d$-rigid subgraph of $G$ 
with at least $d+2$ vertices,
which is maximal with respect to inclusion, and put $X=V(G)\sm V(G')$.
 Then 
$1\leq |X|\leq 4$. 
If some vertex in $X$ had at least $d$ neighbours in $G'$, then we could create a larger $\scrr_d$-rigid subgraph by performing a 0-extension.
Hence each $x\in X$ has at most $d-1$ neighbours in $G'$. Since $G$ has minimum degree at least $d+1$, each $x\in X$ has at least two neighbours in $X$ and we have $3\leq |X|\leq 4$. 

\noindent \smallskip

\noindent \textbf{\boldmath Subcase 1.1:} $|X|=3$.
Then $G[X]=K_3$. In addition, $G'$ is a minimally rigid graph on $d+2$ or $d+3$ vertices so either $G'=K_{d+2}-e$ for some edge $e$, or $G'=K_{d+3}-\{e,f,g\}$ for some  edges $e,f,g$ which are not all incident with the same vertex. If $|N_G(X)|\geq d$ then we could construct an $\scrr_d$-rigid spanning subgraph of $G$ by Lemma \ref{lem:rigunion}.
Hence $|N_G(X)| =d-1$. Since $G$ does not contain a copy of $K_{d+2}$, at least one edge, say $e$, with its end-vertices in $N_G(X)$ is missing from $G$. This gives $G=B_{d,d-1}$ when $G'=K_{d+2}-e$,
so we must have  $G'=K_{d+3}-\{e,f,g\}$. Since $G\not\in \scrb_{d,d-1}^+$, $f$ and $g$ have a common end-vertex $u$, and are both have at least one endvertex in $V(G')\sm N_G(X)$. Since $\delta(G)\geq d+1$, we must have $u\in 
N_G(X)$. Then  the graph obtained from $G$ by deleting all the edges from $u$ to its neighbours in $V(G')\setminus N_G(X)$ is a copy of $B_{d,d-2}$ in $G$. This contradicts the choice of $G$

\smallskip

\noindent \textbf{\boldmath Subcase 1.2: $|X|=4$}.
Then $C_4\subseteq G[X]\subseteq K_4$ and $|V(G')|=d+2$. Since $G'$ is minimally rigid, we have $G'= K_{d+2}-e$ for some edge $e$. 

\begin{claim}\label{cla:Xneighbours}
$N_G(X)=V(G')$.
\end{claim}

\begin{proof}[Proof of claim]
Suppose, for a contradiction, that $N_G(X)\neq V(G')$. Let $Y=X\cup N_G(X)$. Then $G[Y]$ is a proper subgraph of $G$ so is $\scrr_d$-independent.
If $G[N_G(X)]$ was complete, then $G$ would be $\scrr_d$-independent by Lemma~\ref{lem:intbridge}\ref{it:intbridge:indep},
since $G=G'\cup G[Y]$, $G'$ and $G[Y]$ are $\scrr_d$-independent, and
$G'\cap G[Y]=G[N(X)]$ is complete. 
Hence $G[N_G(X)]$ is not complete.  Since $G'= K_{d+2}-e$, this implies that both end-vertices of $e$ belong to $N_G(X)$.
Choose a vertex $w\in V(G')\sm N_G(X)$ and an edge $f$ of $G'$ which is incident with $w$. Consider the graph $G''=G+e-f$.

Suppose $G''[Y]$ is $\scrr_d$-independent. Since $G''[N_G(X)]$ induces a complete graph, we can use Lemma~\ref{lem:intbridge}\ref{it:intbridge:indep}
as above to deduce that $G''$ is $\scrr_d$-independent.
Then $G'+e= K_{d+2}$ is the unique $\scrr_d$-circuit  in $G''+f$ and hence  $G=G''+f-e$ is $\scrr_d$-independent.
This  contradicts the choice of $G$.  Hence $G''[Y]$ is $\scrr_d$-dependent.

Let $C$ be an $\scrr_d$-circuit in $G''[Y]$. Since $G''-e=G-f$ is a proper subgraph of $G$ and $G$  is an $\scrr_d$-circuit, we have $e\in E(C)$.
We also have $w\not\in V(C)$ since $w\not\in Y$.

Suppose $C=B_{d,d-1}$. Then $V(C)= V(G'')\setminus \{w\}=V(G)\setminus\{w\}$. 
Since $E(C)\setminus E(G)=\{e\}$, we may apply a 1-extension to $C$ by adding $w$ and its $d+1$ neighbours in $G$ and deleting $e$, to obtain a spanning subgraph of $G$.
By Lemma \ref{lem:banana_1ext}, this spanning subgraph of $G$ is either $\scrr_d$-rigid (implying that $G$ is $\scrr_d$-rigid) or it is a member of $\scrb_{d,d-1}^+$.
Both of these possibilities contradict the choice of $G$. Thus $C\neq B_{d,d-1}$ and the minimality of $G$ now implies that $C$ is rigid.

Since $G'+e= K_{d+2}$ and  $e\in E(C)\cap E(G'+e)$, the matroid circuit elimination axiom implies that  $(C-e)\cup G'$ is $\scrr_d$-dependent.
Since $(C-e)\cup G'\subseteq G$, we must have $(C-e)\cup G'=G$. This implies that $X$ and all edges of $G$ incident to $X$ are contained in $C$. 
Thus $N_G(X)\subset V(C)$. If $|N_G(X)|\geq d$, then $G=G'\cup (C-e)$ would be  rigid by Lemma~\ref{lem:intbridge}\ref{it:intbridge:rig}. 
Hence $|N_G(X)|\leq d-1$. If $|N_G(X)|=d-2$, then $C=K_{d+2}$ and $G=B_{d,d-2}$. Hence  $|N_G(X)|=d-1$.
Then $C=K_{d+3}-f-g$ for two non-adjacent edges $f,g$ and $G\in \scrb_{d,d-1}^+$.
This contradicts the choice of $G$ and completes the proof of  the claim.
\end{proof}

Suppose $G[X]=C_4$. Since $\delta(G)=d+1$ and no vertex of $X$ has more than $d-1$ neighbours in $G'$,  each vertex of $X$ has degree $d+1$ in $G$.
Claim~\ref{cla:Xneighbours} and the facts that $|N_G(X)|=|V(G')|=d+2$ and each vertex of
$X$ has $d-1$ neighbours in $V(G')$, imply that there exists a vertex  $u\in X$ such that $|N_G(X-u)\cap V(G')|\geq d$.
We can  perform a 1-reduction of $G$ which deletes $u$ and adds an edge between the two neighbours of $u$ in $X$. We can now apply Lemma \ref{lem:rigunion} to the resulting graph $H$ on $d+5$ vertices to deduce that $H$ is $\scrr_d$-rigid.
This implies that $G$ is $\scrr_d$-rigid and contradicts the choice of $G$.
Hence $G[X]\neq C_4$.

Suppose $G[X]=C_4+f$ for some edge $f=wx$. Then $w$ and $x$  have  degree $d+1$ or $d+2$ in $G$ and the vertices in $X\sm \{w,x\}$ have degree $d+1$.
If  $d_G(w)=d_G(x)=d+2$ then $G$ would have more than $d|V|-\binom{d+1}{2}$ edges, so could not be a flexible $\scrr_d$-circuit. Hence we may
assume that  $d_G(w)=d+1$. Construct $H$ from $G$ by performing a 1-reduction which deletes $w$ and adds an edge between the two non-adjacent neighbours of $w$ in
$X$. If  $d_G(x)=d+1$, then $x$ would have degree $d$ in $H$ and we could reduce $H$ to $G'$ by recursively deleting the remaining three vertices of $X$ beginning with $x$,  so that every deleted vertex has degree at most $d$.
Since $G'$ is $\scrr_d$-independent this would imply that $G$ is $\scrr_d$-independent and contradict the choice of $G$.
Hence  $d_G(x)=d+2$.  We can now apply Lemma \ref{lem:rigunion} to deduce that either $H$ is $\scrr_d$-rigid
or $|N_G(X-w)\cap V(G')|=d-1$ and $H$ is $B_{d,d-1}$. The first alternative would imply that $G$ is $\scrr_d$-rigid, and
the second alternative would imply that either $G$ is $\scrr_d$-rigid or $G\in \scrb_{d,d-1}^+$ by Lemma \ref{lem:banana_1ext}. Both alternatives 
 contradict the choice of $G$.
 
Hence $G[X]\neq C_4$.
Then each vertex in $X$ has degree $d+1$ or $d+2$ in $G$. In addition, at most two vertices  of $X$ can have degree $d+2$ in $G$, otherwise $G$ would have more than $d|V|-\binom{d+1}{2}$ edges and could not be a flexible circuit.
Let $\hat G$ be obtained from $G$ by adding edges from vertices in $X$ to vertices in $G'$ in such a way that $X$ has exactly two  vertices of degree $d+1$ and exactly two  vertices of degree $d+2$ in $\hat G$.
We will show that $G$ is $\scrr_d$-independent by proving that $\hat G$ is minimally $\scrr_d$-rigid.

Since $N_{\hat G}(X)=V(G')$ by Claim \ref{cla:Xneighbours}, we may choose vertices $x,y\in X$ such that $x$ has degree $d+1$,
$y$ has degree $d+2$ and some vertex $w\in V(G')$ is a neighbour of $x$ in $\hat  G$ but not $y$. Let $X=\{x,y,z,t\}$ where $z$ has degree $d+2$ and $t$ has degree $d+1$ in $\hat G$.
We can construct $\hat G$ from $G'$ by first performing a 0-extension which adds $y$ and all edges from $y$ to its neighbours in $G'$ as well as to $w$,
then add $z$ and then $t$ by successive 0-extensions, and finally add $x$ by a 1-extension which removes the edge $yw$. see Figure \ref{fig:case1}. Since $G'$ is minimally $\scrr_d$-rigid this implies that $\hat G$ is also minimally $\scrr_d$-rigid. This contradicts the fact that $G$ is an $\scrr_d$-circuit and completes the proof of Case 1.

\begin{figure}
\begin{center}
    \begin{tikzpicture}
      \draw[genericgraph,pattern=north east lines,pattern color=black!10!white] (-0.6,0) circle[x radius=0.8cm, y radius=1.75cm];
			\coordinate (z) at (1.75,0.5) {};
			\coordinate (x) at (1.75,-0.5);
			\node[vertex,rotate around=-60:(z),label={[labelsty]60:$y$}] (y) at (x) {};
			\node[vertex,label={[labelsty]left:$w$}] (cx) at (-0.55,-1.2) {};
			\node[] (cy) at (-1.2,-0.35) {};
			\draw[edge] (y)to node[labelsty,above,pos=0.9] {$d-1$} (cy);
			\draw[edge] (y)edge($(cy)-(0,0.2)$) (y)edge($(cy)-(0,0.4)$) (y)edge(cx);
    \end{tikzpicture}
    \quad
    \begin{tikzpicture}
      \draw[genericgraph,pattern=north east lines,pattern color=black!10!white] (-0.6,0) circle[x radius=0.8cm, y radius=1.75cm];
			\node[vertex,label={[labelsty]above:z}] (z) at (1.75,0.5) {};
			\coordinate (x) at (1.75,-0.5);
			\node[vertex,rotate around=-60:(z),label={[labelsty]60:$y$}] (y) at (x) {};
			\node[vertex,rotate around=60:(y),label={[labelsty]60:$t$}] (t) at (z) {};
			\node[vertex,label={[labelsty]left:$w$}] (cx) at (-0.55,-1.2) {};
			\node[] (cz) at (-1,0.5) {};
			\node[] (cy) at (-1.2,-0.35) {};
			\node[] (ct) at (-0.5,1.5) {};	
			\draw[edge] (y)edge(z) (y)edge(t) (z)edge(t);
			\draw[edge] (y)to node[labelsty,above,pos=0.9] {$d-1$} (cy);
			\draw[edge] (y)edge($(cy)-(0,0.2)$) (y)edge($(cy)-(0,0.4)$) (y)edge(cx);
			\draw[edge] (z)to node[labelsty,below,pos=0.75] {$d-1$} (cz);
			\draw[edge] (z)edge($(cz)+(0,0.2)$) (z)edge($(cz)+(0,0.4)$);
			\draw[edge] (t)to node[labelsty,above] {$d-2$} (ct);
			\draw[edge] (t)edge($(ct)-(0,0.2)$) (t)edge($(ct)-(0,0.4)$);
    \end{tikzpicture}
    \quad
    \begin{tikzpicture}
      \draw[genericgraph,pattern=north east lines,pattern color=black!10!white] (-0.6,0) circle[x radius=0.8cm, y radius=1.75cm];
			\node[vertex,label={[labelsty]above:$z$}] (z) at (1.75,0.5) {};
			\node[vertex,label={[labelsty]below:$x$}] (x) at (1.75,-0.5) {};
			\node[vertex,rotate around=-60:(z),label={[labelsty]60:$y$}] (y) at (x) {};
			\node[vertex,rotate around=60:(y),label={[labelsty]60:$t$}] (t) at (z) {};
			\node[vertex,label={[labelsty]left:$w$}] (cx) at (-0.55,-1.2) {};
			\node[] (cz) at (-1,0.5) {};
			\node[] (cy) at (-1.2,-0.35) {};
			\node[] (ct) at (-0.5,1.5) {};
			\draw[edge] (x)edge(y) (x)edge(z) (x)edge(t) (y)edge(z) (y)edge(t) (z)edge(t);
			\draw[edge] (x)to node[labelsty,below] {$d-2$} ($(cx)-(0,0.4)$);
			\draw[edge] (x)edge($(cx)-(0,0.2)$) (x)edge(cx);
			\draw[edge] (y)to node[labelsty,above,pos=0.9] {$d-1$} (cy);
			\draw[edge] (y)edge($(cy)-(0,0.2)$) (y)edge($(cy)-(0,0.4)$);
			\draw[edge] (z)to node[labelsty,below,pos=0.75] {$d-1$} (cz);
			\draw[edge] (z)edge($(cz)+(0,0.2)$) (z)edge($(cz)+(0,0.4)$);
			\draw[edge] (t)to node[labelsty,above] {$d-2$} (ct);
			\draw[edge] (t)edge($(ct)-(0,0.2)$) (t)edge($(ct)-(0,0.4)$);
    \end{tikzpicture}
\end{center}
\caption{Construction of $\hat G$ in the proof of Case 1.}
\label{fig:case1}
\end{figure}

\medskip
\noindent
\textbf{\boldmath Case 2: $\delta(G)\geq d+2$}.

Choose $v\in V$ with $d(v)=\Delta(G)$. If $G-v$ was $\scrr_{d-1}$-independent then $G$ would be $\scrr_d$-independent by Lemma \ref{lem:coning}. This is impossible since $G$ is an $\scrr_d$-circuit. Hence $G-v$ contains an $\scrr_{d-1}$-circuit $C$.
By the minimality of $d$, $C$ is $\scrr_{d-1}$-rigid or $C\in \{B_{d-1,d-2}, B_{d-1,d-3}\}\cup \scrb_{d-1,d-2}^+$. 

\begin{claim}\label{cla:G-v}
$G-v$ is $\scrr_{d-1}$-rigid.
\end{claim}

\begin{proof}[Proof of Claim]
We first consider the case when  $C\in \{B_{d-1,d-3}\}\cup \scrb_{d-1,d-2}^+$. 
Then $C$ has $d+5$ vertices so is a spanning subgraph of $G-v$.
We have seen that 
every graph  in $\scrb_{d-1,d-2}^+$ has one degree of freedom and has three vertices of degree $d$ on the smaller  side of its $(d-2)$-separation,
and that  $B_{d-1,d-3}$ has three degrees of freedom and has four vertices of degree $d$ on each side of its $(d-3)$-separation.
These observations  and the facts that $\delta(G-v)\geq d+1$ and each nontrivial infinitesimal motion of a generic realisation of $C$ is an infinitesimal rotation about the afffine subspace which contains its separating set of size $d-2$, respectively $d-3$, imply that we can add edges of $G-v$ to $C$ which cross the separating set to make it $\scrr_{d-1}$-rigid. Hence $G-v$ is $\scrr_{d-1}$-rigid.

We next consider the case when $C=B_{d-1,d-2}$. If $C$ is a spanning subgraph of $G$ then we can procceed as in the previous paragraph to deduce that $G-v$ is $\scrr_{d-1}$-rigid. So we may assume that this is not the case. Then $(G-v)\setminus C$ has exactly one vertex $u$.
Since $d_{G-v}(u)\geq d+1$, $G-v$ is $\scrr_{d-1}$-rigid unless all neighbours of $u$ belong to the same copy of $K_{d+1}-e$ in $B_{d-1,d-2}$.
Suppose the second alternative occurs and let $H$ be the spanning subgraph of $G-v$ obtained by adding $u$ and all its incident edges to  $B_{d-1,d-2}$. Then $H$ has one degree of freedom and the smaller side of the $(d-2)$-separation of $H$ contains vertices which have degree $d$ in $H$ and degree at least $d+1$ in $G-v$.
We can now add an edge of $G-v$ to $H$ which crosses its $(d-2)$-separator to make it $\scrr_{d-1}$-rigid. Hence $G-v$ is $\scrr_{d-1}$-rigid.

It remains to consider the case when $C$ is $\scrr_{d-1}$-rigid. Then $|V(C)|\geq d+1$. Let $H$ be a maximal $\scrr_{d-1}$-rigid subgraph of $G-v$ containing $C$.
Suppose $H\neq G-v$ and note that 
$(G-v)-H$ has at most $4$ vertices.
Since each vertex of $(G-v)-H$ has at most $d-2$ neighbours in $H$ and $\delta(G-v)\geq d+1$ we have  $(G-v)-H=K_4$ and $H=C=K_{d+1}$.
We can now apply Lemma \ref{lem:rigunion} to a minimally rigid spanning subgraph of $H$, and to each $K_3$ in $(G-v)-H$, in order to deduce that all vertices of $(G-v)-H$ are adjacent to the same set of $d-2$ vertices of $H$. This cannot occur since every vertex of $H$ which is not joined to a vertex of $G-v-H$ would have degree at most $d+1$ in $G$,  contradicting the assumption of Case 2.
Hence $H=G-v$ and $G-v$ is $\scrr_{d-1}$-rigid.
\end{proof}

Let $(G-v)^*$, respectively $C^*$, be obtained from $G-v$, respectively $C$, by adding $v$ and all edges from $v$ to the vertices of $G-v$, respectively $C$.
Then $(G-v)^*$ is $\scrr_{d}$-rigid  by Claim~\ref{cla:G-v} and Lemma~\ref{lem:coning}, and,  when $C$ is $\scrr_{d-1}$-rigid, $C^*$ is an $\scrr_{d}$-circuit
by Lemma \ref{lem:coning}\ref{it:coning:circ}.

Let $S$ be the set of all edges of $G^*$ which are not in  $G$. Since $C^*$ is rigid or $C^*\in \{B_{d,d-1}, B_{d,d-2}\}\cup \scrb_{d,d-1}^+$, $C^*$ is not an  $\scrr_{d}$-circuit in $G$. Hence $E(C^*)\cap S\neq \emptyset$. If $|S|=1$, say
$S=\{f\}$, then $G=(G-v)^*-f$ would be $\scrr_{d}$-rigid since $(G-v)^*$ is $\scrr_{d}$-rigid and $f\in E(C^*)$. Hence $|S|\geq 2$ and $\Delta(G)=d(v)\leq |V|-3$.  Let $\bar G$ be the complement of $G$. 

Suppose  $|V|\leq d+5$. Then, since $d+2\leq \delta(G)\leq \Delta(G)\leq |V|-3$, we have $|V|=d+5$ and $G$ is $(d+2)$-regular. This implies that  $\bar G$ is a 2-regular graph on $d+5\geq 8$ vertices and we may choose two non-adjacent vertices $v_1,v_2$  with no common neighbours in $\bar G$.   Then $v_1v_2\in E$ and $|N_G(v_1)\cap N_G(v_2)|=d-1$. We can use the facts that $G$ is  
$d$-sparse,  $(d+2)$-regular and $|V|=d+5$ to deduce that    
$G/v_1v_2$ is $d$-sparse. (If not, then some set $X\subseteq V(G/v_1v_2)$ induces more that $d|X|-\binom{d+1}{2}$ edges.
Then $|X|\geq d+2$ and the fact that each vertex of $V(G/v_1v_2)\sm X$ has degree at least $d+1$ implies that $G/v_1v_2$ has more that  $d|V(G/v_1v_2)|-\binom{d+1}{2}$ edges.
This contradicts the fact that $G$ is a flexible $\scrr_{d}$-circuit so has at most $d|V(G)|-\binom{d+1}{2}$ edges.)
Since $|V(G/v_1v_2)|=d+4$, $G/v_1v_2$ has no flexible $\scrr_d$-circuits by the minimality of $G$. Hence $G/v_1v_2$ is  $\scrr_{d}$-independent.
Since $|N_G(v_1)\cap N_G(v_2)|=d-1$, we can now use Lemma \ref{lem:vsplit} to deduce that $G$ is $\scrr_d$-independent and contradict the choice of $G$. 

Hence $|V|= d+6$. Since $\delta(G)\geq d+2$ and $\Delta(G)\leq d+3$ we have 
$\delta(\bar G)\geq 2$ and $\Delta(\bar G)\leq 3$. We can now complete the proof of the theorem by considering three subcases.

\smallskip

\noindent \textbf{\boldmath Subcase 2.1: $\delta(\bar G)= 2$ and $\Delta(\bar G)= 3$}.
 In this case there exist two vertices $x,y\in V$ with $d_{\bar G}(x)=2$, $d_{\bar G}(y)=3$ and $\dist_{\bar G}(x,y)\geq 3$ by Lemma \ref{lem:deg23}. 
Hence $|N_G(x)\cap N_G(y)|=d-1$ and we can deduce as in the previous paragraph that
$G/xy$ is $d$-sparse. 

Suppose  $G/xy$ contains an $\scrr_d$-circuit. Then $G/xy=B_{d,d-1}$ by the minimality and $d$-sparsity of $G$. Since  $B_{d,d-1}$ has $d-3$ vertices of degree $d+4$ and
only the vertex obtained by contracting $xy$ has degree $d+4$ in $G/xy$ 
we must have $d=4$. And when $d=4$, the fact that $\delta(G)=6$ would imply that $x$ and $y$ are adjacent in $G$ to each of the six vertices of degree five in $G/xy$.  Since they are also adjacent to each other this contradicts the fact that $d_G(y)=d+2=6$.

Hence $G/v_1v_2$ is  $\scrr_{d}$-independent. Since $|N_G(v_1)\cap N_G(v_2)|=d-1$, we can now use Lemma \ref{lem:vsplit} to deduce that $G$ is $\scrr_d$-independent and contradict the choice of $G$. 
\smallskip

\noindent \textbf{\boldmath Subcase 2.2: $\bar G$ is 2-regular}.
 In this case we have $|S|=2$ and $G$ is $(d+3)$-regular. The fact that $(G-v)^*$ is $\scrr_d$-rigid  and contains at least two  $\scrr_d$-circuits ($G$ and $C^*$) tells us that $|E((G-v)^*)|\geq d|V(G)|-\binom{d+1}{2}+2$. Since $|E|=|E((G-v)^*)|-|S|$ and $G$ is $d$-sparse this gives
\begin{equation*}
 \frac{(d+3)(d+6)}{2}=|E|=d|V(G)|-\binom{d+1}{2}=\frac{d(d+11)}{2}.
\end{equation*}
This implies that $d=9$ and $|V|=15$. We can now use Lemma~\ref{lem:comp}\ref{it:comp:12reg} to deduce that $G$ is 
$\scrr_9$-independent, contradicting the fact that $G$ is an $\scrr_9$-circuit. 

\smallskip

\noindent \textbf{\boldmath Subcase 2.3: $\bar G$ is 3-regular}.
 In this case we have $|S|=3$ and $G$ is $(d+2)$-regular.  Since $(G-v)^*$ is $\scrr_d$-rigid and contains at least two  $\scrr_d$-circuits we can deduce as in the previous subcase that  $|E(G)|\geq d|V|-\binom{d+1}{2}-1$. The fact that $G$ is $d$-sparse now gives
\begin{equation*}
 \frac{(d+2)(d+6)}{2}=|E|=d|V|-\binom{d+1}{2}-\alpha=\frac{d(d+11)}{2}-\alpha
\end{equation*}
for some $\alpha =0, 1$. This implies that $\alpha=0$, $d=4$ and $|V|=10$. 
We can now use Lemma~\ref{lem:comp}\ref{it:comp:6reg} to deduce that $G$ is 
$\scrr_4$-independent, contradicting the fact that $G$ is an $\scrr_4$-circuit. 
\phantom{p}
\hfill
$\Box$

\subsection{Proof of Corollary \ref{cor:edges}}
The corollary follows immediately from Theorem \ref{thm:vertex} if $|V|\leq d+6$. Since $\delta(G)\geq d+1$ we have $|E|>d(d+9)/2$ when either $|V|\geq d+8$, or $|V|= d+7$ and $\delta(G)\geq d+2$. Hence we may assume that $|V|= d+7$ and $\delta(G)= d+1$. Choose a vertex $v$ with $d(v)=d+1$.
Then $v$ has two non-adjacent neighbours $v_1,v_2$ since otherwise $G$ would contain the 
rigid $\scrr_d$-circuit $K_{d+2}$. Let $H=G-v+v_1v_2$. If $H$ was $\scrr_d$-independent then $G$ would be $\scrr_d$-independent  by Lemma \ref{lem:01ext}. Hence $H$ contains an $\scrr_d$-circuit $C$. If $C$ is flexible then Theorem \ref{thm:vertex} implies that $C\in \{B_{d,d-1},B_{d,d-2}\}\cup \scrb_{d,d-1}^+$  and hence $|E|>|E(C)|\geq d(d+9)/2$. Thus we may assume that $C$ is $\scrr_d$-rigid. Then  $C-v_1v_2$ is an $\scrr_d$-rigid subgraph with at least $d+2$ vertices. Let $X=V(G)\sm V(C)$ and let $E(X,V\sm X)$ be the set of edges with one endvertex in $X$ and one in $V\sm X$.
Then $1\leq |X|\leq 5$. Since $\delta(G)= d+1$ and $|X|\leq 5$ we have 
\begin{align*}
|E|&=  |E(C-v_1v_2)|+|E(X)|+|E(X,V\sm X)|\\
&\geq d|V\sm X|-\binom{d+1}{2}+\binom{|X|}{2}+|X|(d+1-|X|+1)\\
&= d|V|-\binom{d+1}{2}-\frac{|X|(|X|-3)}{2}\\
&\geq \frac{d(d+13)}{2}-5.
\end{align*}
We can now use the fact that $d\geq 3$ to deduce that $|E|> d(d+9)/2$.
\qed

\subsection{Proof of Corollary \ref{cor:tight}}
Suppose $G=(V,E)$ is $\scrr_d$-rigid. Then $r_d(G)=d|V|-\binom{d+1}{2}$. Let $H=(V,F)$ be a maximal $\scrr_d$-independent subgraph of $G$. Then
$|F|=r_d(H)=d|V|-\binom{d+1}{2}$, $H$ is $d$-tight and $B_{d,d-1}, B_{d,d-2}\not\subseteq H$ by Lemma \ref{lem:banana}. In addition
 $H$ is $d$-connected by Lemma \ref{lem:intbridge}\ref{it:intbridge:rank}.

Conversely, suppose that $G$ has a spanning subgraph $H=(V,F)$ which satisfies the hypotheses of the statement. Since $H$ is $d$-tight, it is $d$-sparse and hence does not
contain any $\scrr_d$-rigid circuits. If $C\subseteq H$ for some $\scrb_{d,d-1}^+$ then we would have $C=H$ since $C$ is $d$-tight and $|V(C)|=d+6\geq |V(G)|$. This would contradict the hypothesis that $H$ is  $d$-connected so $H$ contains no graph in $\scrb_{d,d-1}^+$. Theorem \ref{thm:vertex} combined with   the hypothesis that $B_{d,d-1}, B_{d,d-2}\not\subseteq H$ now implies that $H$ does not contain any flexible $\scrr_d$-circuits. Hence $H$ is $\scrr_d$-independent. Since $H$ is $d$-tight, it is $\scrr_d$-rigid. This and the hypothesis that $H$ is a spanning subgraph of $G$ imply that $G$ is $\scrr_d$-rigid.
\qed

\section{Closing Remarks}
We briefly consider some possible extensions of our results.
\subsection{Generalised 2-sums}

Let $G=(V,E)$, $G_1=(V_1,E_1)$ and $G_2=(V_2,E_2)$ be graphs. We say that $G$ is a {\em $t$-sum of $G_1,G_2$ along an edge $e$} if $G=(G_1\cup G_2)-e$, $G_1\cap G_2=K_t$ and $e\in E_1\cap E_2$. We conjecture that  Lemma \ref{lem:2-sum} can be extended to $t$-sums.

\begin{conj} \label{con:t-sum} Suppose that
 $G$ is a $t$-sum of $G_1,G_2$ along an edge $e$ for some $2\leq t\leq d+1$.  Then $G$ is an $\scrr_d$-circuit if and only if $G_1,G_2$ are $\scrr_d$-circuits. 
\end{conj} 
 
 Our proof technique for  Lemma \ref{lem:2-sum} gives the following partial result.

\begin{lemma}\label{lem:t-sum} 
Let $G=(V,E)$, $G_1=(V_1,E_1)$ and $G_2=(V_2,E_2)$ be graphs such that $G$ is a $t$-sum of $G_1,G_2$ along an edge $e$ for some $2\leq t\leq d+1$.
\begin{enumerate}
	\item\label{it:t-sum:Gcirc} If $G$ is an $\scrr_d$-circuit, then $G_1$ and $G_2$ are both $\scrr_d$-circuits.
	\item\label{it:t-sum:G12circ} If $G_1$ and $G_2$ are both $\scrr_d$-circuits, then $G$ contains a unique $\scrr_d$-circuit $G'$ and\break $E\sm (E_1\cap E_2)\subseteq E(G')$.
\end{enumerate}
\end{lemma}

\begin{proof}
\ref{it:t-sum:Gcirc} If  $G_1$ and $G_2$ are both  $\scrr_d$-independent, then  Lemma \ref{lem:intbridge}\ref{it:intbridge:indep} implies that $G_1\cup G_2$ is $\scrr_d$-independent. This contradicts the facts that $G$ is an $\scrr_d$-circuit and $G\subseteq G_1\cup G_2$. If  exactly one of $G_1$ and $G_2$, say $G_1$, is  $\scrr_d$-independent then $e$ belongs to the unique  $\scrr_d$-circuit in $G_2$ and  Lemma \ref{lem:intbridge}\ref{it:intbridge:indep} gives  $r_d(G)=r_d(G+e)=|E_1|+|E_2|-\binom{t}{2}-1=|E|$.
 This again contradicts the hypothesis that $G$ is an $\scrr_d$-circuit.  Hence
$G_1$ and $G_2$ are both $\scrr_d$-dependent. Then the matroid circuit elimination axiom  combined with the fact that $G$ is an $\scrr_d$-circuit imply that $G_1$ and $G_2$ are both $\scrr_d$-circuits.  

\medskip
\noindent
\ref{it:t-sum:G12circ}
The circuit elimination axiom implies that $G$ is $\scrr_d$-dependent and hence that $G$ contains an $\scrr_d$-circuit $G'=(V',E')$.
Since $G_i-e$ is $\scrr_d$-independent for $i=1,2$, we have $E'\sm E_i\neq \emptyset$.
Let $G_i'$ be obtained from $G_i\cap G'$ by adding an edge between every pair of non-adjacent vertices in $V'\cap V_1\cap V_2$.
If $G_i'$ is a proper subgraph of $G_i$ for $i=1,2$ then each $G_i'$ is $\scrr_d$-independent and
we can use Lemma~\ref{lem:intbridge}\ref{it:intbridge:indep} to deduce that $G_1'\cup G_2'$ is $\scrr_d$-independent.
This gives a contradiction since $G'\subseteq G_1'\cup G_2'$. Relabelling if necessary we have $G_1'=G_1$.
If $G_2'\neq G_2$ then we may deduce similarly that $G_1'\cup G_2'-e$ is independent.
This again gives a contradiction since $G'\subseteq G_1'\cup G_2'-e$. Hence $G_2'=G_2$.
It remains to show uniqueness. For $i=1,2$, let $B_i$ be a base of $\scrr_d(G_i)$ which contains $E(G_1)\cap E(G_2)$ .
Then $|B_i|=|E_i|-1$ and
Lemma \ref{lem:intbridge}\ref{it:intbridge:indep} gives  
\begin{equation*}
r_d(G)=r_d(G_1\cup G_2-e)=r_d(G_1\cup G_2)=|B_1|+|B_2|-\binom{t}{2}=|E|-1.
\end{equation*}
Hence, $G$ contains a unique $\scrr_d$-circuit.
\end{proof}

Conjecture \ref{con:t-sum} holds when $t=d+1$ and $G_1,G_2$ are both globally rigid in $\R^d$ by a result of Connelly \cite{Con}. It also holds when $d=2$ and $t=3$ by a result of  Jord\'an \cite[Theorem 3.6.15]{J16}.

\subsection{Highly connected flexible circuits}
Bolker and Roth \cite{BR} determined $r_d(K_{s,t})$ for all complete bipartite graphs $K_{s,t}$. Their result implies that  $K_{d+2,d+2}$ is a $(d+2)$-connected  $\scrr_d$-circuit  for all $d\geq 3$ and is flexible when $d\geq 4$, see \cite[Theorem 5.2.1]{GSS}. 
We know of no $(d+3)$-connected flexible $\scrr_d$-circuits  and it is tempting to conjecture that they do not exist.

For the case when $d=3$, Tay \cite{Tay} gives examples of  $4$-connected flexible $\scrr_3$-circuits and Jackson and Jord\'an \cite{JJ06} conjecture that all $5$-connected $\scrr_3$-circuits are rigid. An analogous statement has recently been verified for circuits in the closely related $C^1_2$-cofactor matroid by Clinch, Jackson and Tanigawa \cite{CJT2}.

\subsection{Extending Theorem \ref{thm:vertex}}
We saw in the previous subsection that $K_{d+2,d+2}$ is a flexible  $\scrr_d$-circuit  with $2d+4$ vertices for all $d\geq 4$.
We can obtain a $(d+2)$-connected flexible $\scrr_d$-circuit on $d+8$ vertices by recursively applying the coning operation to the flexible $\scrr_4$-circuit $K_{6,6}$ and then applying Lemma \ref{lem:coning}.
This suggests that it may be difficult to extend Theorem \ref{thm:vertex} to graphs on $d+8$ vertices, but it is conceivable that all flexible $\scrr_d$-circuits on $d+7$ vertices have the form $(G_1\cup G_2)-S$ where $G_i\in \{K_{d+2}, K_{d+3},K_{d+4}\}$, $G_1\cap G_2\in \{K_{d-3}, K_{d-2},K_{d-1}\}$ and $S$ is a suitably chosen set of edges.

For the case when $d=3$, Tay \cite{Tay} gives examples of  $3$-connected flexible $\scrr_3$-circuits with 13 vertices but it is possible that all flexible circuits on at most 12 vertices can be obtained by taking 2-sums of rigid circuits on at most 9 vertices.

\subsection*{Acknowledgement}
\addcontentsline{toc}{section}{Acknowledgement}
The authors would like to thank the referees for their careful reading and helpful comments.
We thank the London Mathematical Society, and the Heilbronn Institute for Mathematical Research, for providing partial financial support through a scheme 5 grant, and a focussed research group grant, respectively. Georg Grasegger was supported by the Austrian Science Fund (FWF): P31888.

\phantomsection

\newpage
\phantomsection
\addcontentsline{toc}{section}{Corrigendum}
\setcounter{theorem}{0}
\begin{center}%
    {\LARGE Corrigendum to `Flexible circuits in the $d$-dimensional rigidity matroid' \par}%
    \vskip 1.5em%
    {\large
      \lineskip .5em%
      \begin{tabular}[t]{c}%
        Georg Grasegger, Hakan Guler, Bill Jackson and Anthony Nixon
      \end{tabular}\par}%
    \vskip 1em%
  \end{center}%

\begin{abstract}
We give a counterexample to Lemma 18(a) and Conjecture 17 in our paper \cite{GGJN} and provide a corrected proof for a weaker version of  Lemma 18(a).
\end{abstract}

We refer the reader to \cite{GGJN,W} for terminology not explicitly defined in this note.
Let $G=(V,E)$, $G_1=(V_1,E_1)$ and $G_2=(V_2,E_2)$ be graphs. We say that $G$ is a {\em $t$-sum of $G_1,G_2$ along an edge $e$} if $G=(G_1\cup G_2)-e$, $G_1\cap G_2=K_t$ and $e\in E_1\cap E_2$. We proved that the 2-sum of $G_1,G_2$ is an $\scrr_d$-circuit if and only if $G_1,G_2$ are $\scrr_d$-circuits in \cite[Lemma 10]{GGJN}, conjectured that this result extends to $t$-sums for all $2\leq t\leq d+1$ in \cite[Conjecture 17]{GGJN}, and claimed to prove a weaker version of this conjecture in \cite[Lemma 18]{GGJN}.  In particular \cite[Lemma 18(a)]{GGJN}  claimed that $G_1,G_2$ are both $\scrr_d$-circuits when their $t$-sum is an $\scrr_d$-circuit.
\begin{figure}[ht]
\begin{center}

\begin{tikzpicture}[]

\node[roundnode] at (0cm,0cm) (u) [] {};
\node[roundnode] at (0:1.5cm) (v1) [] {}
	edge[](u);
\node[roundnode] at (40:1.5cm) (v2) [] {}
	edge[](u)
	edge[](v1);

\node[roundnode] at (80:1.5cm) (v3) [] {}
	edge[](u)
	edge[](v1)
	edge[](v2);
\node[roundnode] at (120:1.5cm) (v4) [] {}
	edge[](u)
	edge[](v3);

\node[roundnode] at (160:1.5cm) (v5) [] {}
	edge[](u)
	edge[](v3)
	edge[](v4);
\node[roundnode] at (200:1.5cm) (v6) [] {}
	edge[](v3)
	edge[](v4)
	edge[](v5);

\node[roundnode] at (240:1.5cm) (v7) [] {}
	edge[](u)
	edge[](v6);
\node[roundnode] at (280:1.5cm) (v8) [] {}
	edge[](u)
	edge[](v6)
	edge[](v7);
\node[roundnode] at (320:1.5cm) (v9) [] {}
	edge[](v1)
	edge[](v2)
	edge[](v3)
	edge[](v6)
	edge[](v7)
	edge[](v8);
\node[] at (270:1.5cm) () [label=below:$G$]{};
\begin{scope}[xshift=6cm]
\node[roundnode] at (0cm,0cm) (u) [] {};
\node[roundnode] at (0:1.5cm) (v1) [] {}
	edge[](u);
\node[roundnode] at (40:1.5cm) (v2) [] {}
	edge[](u)
	edge[](v1);

\node[roundnode] at (80:1.5cm) (v3) [] {}
	edge[](u)
	edge[](v1)
	edge[](v2);

\node[roundnode] at (200:1.5cm) (v6) [] {}
	edge[dashed,thick](u)
	edge[thick](v3);

\node[roundnode] at (240:1.5cm) (v7) [] {}
	edge[](u)
	edge[](v6);
\node[roundnode] at (280:1.5cm) (v8) [] {}
	edge[](u)
	edge[](v6)
	edge[](v7);
\node[roundnode] at (320:1.5cm) (v9) [] {}
	edge[](v1)
	edge[](v2)
	edge[](v3)
	edge[](v6)
	edge[](v7)
	edge[](v8);
\node[] at (270:1.5cm) () [label=below:$G_1$]{};
\end{scope}

\begin{scope}[xshift=6cm]
\begin{scope}[shift=(140:1cm)]
\node[roundnode] at (0cm,0cm) (u) [] {};
\node[roundnode] at (80:1.5cm) (v3) [] {}
	edge[](u);
\node[roundnode] at (120:1.5cm) (v4) [] {}
	edge[](u)
	edge[](v3);

\node[roundnode] at (160:1.5cm) (v5) [] {}
	edge[](u)
	edge[](v3)
	edge[](v4);
\node[roundnode] at (200:1.5cm) (v6) [] {}
	edge[dashed,thick](u)
	edge[thick](v3)
	edge[](v4)
	edge[](v5);

\node[] at (160:1.5cm) () [label=left:$G_2$]{};
\end{scope}
\end{scope}

\end{tikzpicture}

\end{center}
\caption{The graph $G$ is an $\scrr_3$-circuit and is a 3-sum of $G_1,G_2$, along the dashed edge. But $G_1$ is not an $\scrr_3$-circuit since it properly contains the $\scrr_3$-circuit $G_1-f$, where  the bold edge represents $f$.}
\label{fig:1}
\end{figure}
The graphs shown in Figure \ref{fig:1} demonstrate that \cite[Lemma 18(a)]{GGJN} and \cite[Conjecture 17]{GGJN} are false.

The error in the proof of \cite[Lemma 18(a)]{GGJN} occurs in the last sentence of the proof. We can modify this sentence to obtain the following weaker statement. We include the full proof for the sake of completeness.

\begin{lemma}\label{lem:t-sum-c}
Let $G=(V,E)$, $G_1=(V_1,E_1)$ and $G_2=(V_2,E_2)$ be graphs such that $G$ is a $t$-sum of $G_1,G_2$ along an edge $e$ for some $2\leq t\leq d+1$.
 Suppose $G$ is an $\scrr_d$-circuit.  Then, for both $i=1,2$, $G_i$ contains a unique $\scrr_d$-circuit $G_i'$ with $e\in E(G_i')$ and $E(G)=(E(G_1')\cup E(G_2'))-e$.
\end{lemma}

\begin{proof}
 If  $G_1$ and $G_2$ are both  $\scrr_d$-independent, then  \cite[Lemma 9(b)]{GGJN} implies that $G_1\cup G_2$ is $\scrr_d$-independent. This contradicts the facts that $G$ is an $\scrr_d$-circuit and $G\subseteq G_1\cup G_2$. If  exactly one of $G_1$ and $G_2$, say $G_1$, is  $\scrr_d$-independent then $e$ belongs to the unique  $\scrr_d$-circuit in $G_2$ and  \cite[Lemma 9(b)]{GGJN} gives  $r_d(G)=r_d(G+e)=|E_1|+|E_2|-\binom{t}{2}-1=|E|$.
 This again contradicts the hypothesis that $G$ is an $\scrr_d$-circuit.  Hence
$G_1$ and $G_2$ are both $\scrr_d$-dependent. Then the matroid circuit elimination axiom  combined with the fact that $G$ is an $\scrr_d$-circuit imply that, for both $i=1,2$, $G_i$ contains a unique $\scrr_d$-circuit $G_i'$ with $e\in E(G_i')$ and $E(G)=(E(G_1')\cup E(G_2'))-e$.
\end{proof}

We claimed in the paragraph immediately after the proof of \cite[Lemma 18]{GGJN} that two special cases of \cite[Conjecture 17]{GGJN} hold,
but this claim was based on the assumption that \cite[Lemma 18(a)]{GGJN} is valid. The first claim remains true and is given as Lemma \ref{lem:t-sum1-c} below. The second claim, that the $3$-sum of two graphs $G_1,G_2$ along an edge $e$ is an $\scrr_2$-circuit if and only if $G_1,G_2$ are $\scrr_2$-circuits is false. A counterexample is given  in Figure \ref{fig:2}.

\begin{lemma}\label{lem:t-sum1-c}
Let $G$, $G_1$ and $G_2$ be graphs such that $G$ is a $(d+1)$-sum of $G_1,G_2$ along an edge $e$ for some $d\geq 1$.
Suppose $G_1,G_2$ are globally rigid in $\real^d$. Then $G$ is an $\scrr_d$-circuit if and only if $G_1,G_2$ are $\scrr_d$-circuits.
\end{lemma}
\begin{proof}
The hypothesis that $G$ is a $(d+1)$-sum of two globally rigid graphs implies that $G$ is globally rigid in $\real^d$ by \cite[Theorem 11]{C}
and hence is redundantly rigid in $\real^d$ by \cite{H}.
If $G_1,G_2$ are $\scrr_d$-circuits, then $G$ contains a unique $\scrr_d$-circuit by \cite[Lemma 18(b)]{GGJN} and the fact that
$G$ is redundantly rigid in $\real^d$ now implies that $G$ is an $\scrr_d$-circuit.
On the other hand, if $G$ is an $\scrr_d$-circuit, then $G_i$ is
an $\scrr_d$-circuit for both $i=1,2$ by Lemma \ref{lem:t-sum-c} and the hypothesis that $G_i$ is (globally) rigid in $\real^d$.
\end{proof}

\begin{figure}[ht]
\begin{center}

\begin{tikzpicture}[]

\node[roundnode] at (0cm,0cm) (u1) [] {};
\node[roundnode] at (1cm,0cm) (u2) [] {}
	edge[](u1);
\node[roundnode] at (2cm,0cm) (u3) [] {}
	edge[](u2);
\node[roundnode] at (2cm,1cm) (u4) [] {}
	edge[](u3);
\node[roundnode] at (1cm,1cm) (u5) [] {}
	edge[](u4)
	edge[](u1);
\node[roundnode] at (0cm,1cm) (u6) [] {}
	edge[](u5)
	edge[](u1)
	edge[](u2);
\node[roundnode] at (1.5cm,0.5cm) (u7) [] {}
	edge[](u2)
	edge[](u3)
	edge[](u4)
	edge[](u5);
\node[] at (1cm,0cm) () [label=below:$G$]{};

\begin{scope}[xshift=4cm]
\node[roundnode] at (0cm,0cm) (u1) [] {};
\node[roundnode] at (1cm,0cm) (u2) [] {}
	edge[](u1);
\node[roundnode] at (1cm,1cm) (u5) [] {}
	edge[](u1)
	edge[dashed](u2);
\node[roundnode] at (0cm,1cm) (u6) [] {}
	edge[](u5)
	edge[](u1)
	edge[](u2);
\node[roundnode] at (1.5cm,0.5cm) (u7) [] {}
	edge[](u2)
	edge[](u5);
\node[] at (0.75cm,0cm) () [label=below:$G_1$]{};
\end{scope}

\begin{scope}[xshift=6cm]
\node[roundnode] at (1cm,0cm) (u2) [] {};
\node[roundnode] at (2cm,0cm) (u3) [] {}
	edge[](u2);
\node[roundnode] at (2cm,1cm) (u4) [] {}
	edge[](u3);
\node[roundnode] at (1cm,1cm) (u5) [] {}
	edge[](u4)
	edge[dashed](u2);
\node[roundnode] at (1.5cm,0.5cm) (u7) [] {}
	edge[](u2)
	edge[](u3)
	edge[](u4)
	edge[](u5);
\node[] at (1.5cm,0cm) () [label=below:$G_2$]{};
\end{scope}
\end{tikzpicture}

\end{center}
\caption{The graph $G$ is an $\scrr_2$-circuit and is a 3-sum of $G_1$ and $G_2$ along the dashed edge. However $G_1$ is not an $\scrr_2$-circuit since it properly contains the $\scrr_2$-circuit $K_4$.}
\label{fig:2}

\end{figure}

\phantomsection

\end{document}